\documentclass[a4paper]{article}

\usepackage{amssymb}
\usepackage{mathtools}
\usepackage{amsthm}
\usepackage{titlesec}
\usepackage{enumerate}

\let\OLDthebibliography\thebibliography
\renewcommand\thebibliography[1]{
	\OLDthebibliography{#1}
	\setlength{\itemsep}{0pt}
}

\newcommand{\C}{\mathbb C}
\newcommand{\Q}{\mathbb Q}
\newcommand{\R}{\mathbb R}
\newcommand{\Z}{\mathbb Z}
\newcommand{\OG}{\operatorname{O}}
\newcommand{\PS}{\mathcal P}
\newcommand{\HS}{\mathcal H}
\newcommand{\KS}{\mathcal K}
\newcommand{\TS}{\mathcal T}
\newcommand{\SSp}{\mathcal S}
\newcommand{\GS}{\mathcal G}
\newcommand{\g}{\mathfrak g}
\newcommand{\h}{\mathfrak h}
\newcommand{\Aut}{\operatorname{Aut}}
\newcommand{\rk}{\operatorname{rk}}
\newcommand{\hyp}{\operatorname{I\!I_{1,1}}}
\newcommand{\SL}{\operatorname{SL}}
\newcommand{\vac}{\vert 0 \rangle}
\newcommand{\und}{\underline}

\newtheoremstyle{thmbrk}{3pt}{3pt}{\itshape}{}{\bfseries}{}{\newline}{}
\theoremstyle{thmbrk}
\newtheorem{theorem}{Theorem}[section]
\newtheorem{prop}[theorem]{Proposition}

\newtheorem{lemma}[theorem]{Lemma}

\titleformat*{\section}{\bfseries\large}
\titleformat*{\subsection}{\large}

\numberwithin{equation}{section}

\allowdisplaybreaks

\begin{document}
\title{\textbf{Vertex operator expressions for Lie algebras of physical states}}
\author{Thomas Driscoll-Spittler\\
Fachbereich Mathematik,
Technische Universit\"{a}t Darmstadt\\
Darmstadt, Deutschland}
\date{}
\maketitle

\begin{abstract}
\noindent
We study the Lie algebra of physical states associated with certain vertex
operator algebras of central charge $24$. By applying the no-ghost theorem from
string theory we express the corresponding Lie brackets in terms of vertex algebra
operations. In the special case of the Moonshine module this result answers a question
of Borcherds, posed in his paper on the Monstrous moonshine conjecture.
\end{abstract}

\section{Introduction}
The Monstrous moonshine conjecture of Conway and Norton asked for a graded module
of the Monster group $\mathbb{M}$ such that its graded traces are Hauptmoduln for genus-$0$
subgroups of $\SL_2(\R)$. An $\mathbb{M}$-module that has this property for many
conjugacy classes of the Monster group was constructed by Frenkel, Lepowsky and Meurman
in \cite{FLM}. This is the Moonshine module $V^\natural$. In \cite{Bo92} Borcherds 
showed that the graded traces of all elements in the Monster group on this module
are Hauptmoduln. In his proof he made use of the vertex algebra structure of $V^\natural$,
or more precisely, the associated Lie algebra of physical states, the
Monster Lie algebra $\mathfrak{m}$. 
  
More generally, to any vertex operator algebra $V$ of central charge $24$ of CFT-type
with a non-degenerate, invariant bilinear form we can associate a Lie algebra of physical
states $\g(V)$. This Lie algebra can be constructed by a certain quantisation procedure
and carries an action of $\Aut(V)$ as well as a natural $\hyp$-grading
\begin{equation*}
  \g(V)=\bigoplus_{\alpha\in\hyp} \g(V)_\alpha.
\end{equation*}
Here  $\hyp$ denotes the unique even, unimodular lattice of signature $(1,1)$.
The no-ghost theorem states:
  
\emph{For every $0\ne \alpha\in\hyp$ there exists a linear isomorphism
$\eta_\alpha:V_{1-\alpha^2/2}\to\g(V)_{\alpha}$ which preserves the action of $\Aut(V)$.}
  
This version of the no-ghost theorem is due to Borcherds. It was used to determine the
full structure of $\mathfrak{m}$ as a generalised Kac-Moody algebra. A homological
version of the denominator identity of the Monster Lie algebra then implies the complete
replicability of the graded trace functions of the Moonshine module. This proves the
Monstrous moonshine conjecture.
  
Aside from Monstrous moonshine and its generalisations, Lie algebras of physical
states have been applied to relate vertex operator algebras to automorphic forms
and to prove structure results for the corresponding vertex operator algebras.
See \cite{Bo90}, \cite{HS03}, \cite{HS14}, \cite{CKS}, \cite{M21}, \cite{DSSW} and
\cite{SWW} as well as \cite{Mi23} respectively.
  
The no-ghost isomorphisms $\eta_\alpha$ give bilinear maps
\begin{align*}
  \{\cdot,\cdot\}_{\alpha,\beta}: V_{1-\alpha^2/2}\times V_{1-\beta^2/2}&\to 
  V_{1-(\alpha+\beta)^2/2}, \\
  (v,w)&\mapsto \pm\eta_{\alpha+\beta}^{-1}
  ([\eta_\alpha(v),\eta_\beta(w)]),
\end{align*}
where the sign $\pm$ is given by a suitable $2$-cocycle associated with the lattice
$\hyp$. In \cite{Bo92}, Section 15, Borcherds asked (Question $4$) for an explicit
description of those maps in terms of vertex algebra operations of $V$, at least in
the special case of the Moonshine module.
  
In this paper we provide an answer to this question. Our main theorem
(Theorem \ref{Blofeld}) states:
  
\emph{Let $V$ be a vertex operator algebra of central charge $24$ of CFT-type with a
non-degenerate, invariant bilinear form. Choose a primitive isotropic element
$f\in\hyp$ and assume that $\alpha,\beta,\alpha+\beta\notin f^\perp$. Then for
$v\in V_{1-\alpha^2/2}$ and $w\in V_{1-\beta^2/2}$ we have
\begin{align*}
  \{v,w\}_{\alpha,\beta}=
  \sum_{n_1,n_2=0}^\infty\sum_{k=0}^\infty p^{\alpha,\alpha+\beta}_{n_1+n_2+k-(\alpha,\beta)}
  (\imath_{n_1}(v)_k\jmath_{n_2}(w)),
\end{align*} 
with explicitly given operators $p^{\alpha,\alpha+\beta}_{h}$, $\imath_{n}$ and
$\jmath_{m}$ for $h\in\Z$ and $n,m\in\Z_{\geq 0}$.}
  
The proof of this result relies on a description of the no-ghost isomorphisms
in terms of an operator $E$, which was introduced by physicists and first employed
in mathematics by  I. Frenkel.
    
This paper is organised as follows: We recall some well-known facts about lattice
vertex algebras and sketch the covariant quantisation. Then we review the no-ghost
theorem and Borcherds' proof. This exposition is influenced by work of I. Frenkel.
Finally we consider the Lie algebra of physical states of suitable vertex operator
algebras and express the corresponding Lie brackets in terms of vertex algebra operations.
  
\subsection*{Acknowledgements}
This paper is based on results obtained in the author's PhD thesis supervised by
Nils R. Scheithauer. The author wants to thank
Sven M\"oller, Manuel M\"uller, Nils R. Scheithauer and Janik Wilhelm for valuable
comments and fruitful discussions. Furthermore the author acknowledges support by the
LOEWE research unit \emph{Uniformized Structures in Arithmetic and Geometry} and by
the DFG through the CRC \emph{Geometry and Arithmetic of Uniformized Structures},
project number 444845124.

\section{Vertex algebras and the covariant quantisation}
We briefly sketch some basic properties of vertex operator algebras and discuss
lattice vertex algebras. Based on this we introduce the covariant quantisation and
use it to construct Lie algebras of physical states. See \cite{Bo86}, \cite{KAC},
\cite{FBZ}, \cite{FHL}, \cite{Bo92} and \cite{Ju95}. 

A real \emph{vertex algebra} is a real vector space $V$ with a non-zero \emph{vacuum}
$\vac\in V$, a linear \emph{translation operator} $T:V\to V$ and a \emph{vertex operator},
which is a linear map
\begin{equation*}
  Y(\cdot,z):V \to \text{End}(V)[[z^{\pm 1}]],\,
  v\mapsto Y(v,z)=\sum_{n\in\mathbb{Z}}v_nz^{-n-1}.
\end{equation*}
For $v,w\in V$ the vertex operator satisfies
\begin{equation*}
  Y(v,z)w=\sum_{n\in\mathbb{Z}}v_nwz^{-n-1}\in V((z)),
\end{equation*}
where $V((z))$ is the space of formal Laurent series with values in $V$. In addition we
have the \emph{translation axiom}, the \emph{vacuum axiom} and the \emph{locality axiom}.
The latter was originally expressed by Borcherds as a version of the \emph{Borcherds identity}
which states that for $v,w\in V$ and $m,l,k\in \mathbb{Z}$ we have
\begin{align*}
  \sum_{n\geq 0}\binom{m}{n} &(v_{n+l}w)_{m+k-n}\\
  &=\sum_{n\geq 0}\binom{l}{n}((-1)^nv_{m+l-n}w_{k+n}-(-1)^{n+l}w_{k+l-n}v_{m+n}).
\end{align*}
A \emph{conformal vector} is a vector $\omega\in V$ such that its modes
\begin{equation*}
  Y(\omega,z)=\sum_{n\in\mathbb{Z}}L_nz^{-n-2}
\end{equation*}
generate a representation of the Virasoro algebra. More precisely for some $c\in\mathbb{R}$
the modes $L_n$ satisfy
\begin{equation}
  [L_n,L_m]=(n-m)L_{n+m}+\frac{n^3-n}{12}\delta_{n+m,0}c.
\end{equation}
The number $c$ is called the \emph{central charge} of $\omega$. A \emph{conformal
vertex algebra} of \emph{central charge}  $c$ is a vertex algebra $V$ which has a
conformal vector $\omega\in V$ of central charge $c$ such that $L_{-1}=T$ and $L_0$
acts semi-simply with integral eigenvalues on $V$. Hence $V$ is $\Z$-graded with
respect to  $L_0$ and we denote the corresponding eigenspaces by $V_n$ for $n\in\Z$.
A \emph{vertex operator algebra} is a conformal vertex algebra such that the weight
spaces $V_n$ are finite-dimensional and $V_n=0$ for sufficiently small $n$.
Furthermore a vertex operator algebra $V$ is said to be of \emph{CFT-type} if the
$V_n$ vanish for $n<0$ and $V_0=\R\vac$.

A bilinear form $(\cdot,\cdot)$ on a conformal vertex algebra is called \emph{invariant}
if for all $u,v,w\in V$ we have
\begin{equation}
  (Y(u,z)v,w)=(v,Y(e^{zL_1}(-z^{-2})^{L_0}u,z^{-1})w).
\end{equation}
Invariant bilinear forms on vertex operator algebras are symmetric and the space
of such forms is isomorphic to the dual space of $V_0/L_1 V_1$.  Hence a vertex
operator algebra $V$ of CFT-type has a unique invariant bilinear form such that
$(\vac,\vac)=1$. See \cite{FHL} and \cite{Li94}. We denote the orthogonal group with
respect to such a bilinear form by $\OG(V)$ and introduce the group
\begin{equation}\label{AstonMartin}
  G_V= \{g\in \OG(V): L_ng=gL_n\text{ for }n\in\Z\}.
\end{equation}
This is the group of all orthogonal maps on $V$ which commute with the Virasoro
operators. Clearly we have $\Aut(V)\subset G_V$.

We  review real conformal lattice vertex algebras associated with even, non-degenerate
lattices. Let $L$ be an even, non-degenerate lattice. We denote the corresponding real
conformal vertex algebra of central charge $l=\rk(L)$ by $V_L$. Consider the real
commutative Lie algebra $\h=L\otimes\R$ and its \emph{affinisation}
\begin{equation*}
 \hat\h = \h\otimes \R[t,t^{-1}]+ \R K,
\end{equation*}
with a central element $K$. For $h\in \h$ we set $h(n)=h\otimes t^n$, with $n\in\Z$.
We have the relation $[h(n),h'(m)] = n\delta_{n+m,0}(h,h')K$ for all $h,h'\in\h$ and
$n,m\in\Z$. We set $\hat{\h}_-=\h\otimes t^{-1}\R[t^{-1}]$.
The space $S(\hat{\h}_-)$ carries an action of $\hat\h$ called the
\emph{Fock representation}. For a choice of $2$-cocycle $\epsilon:L\times L \to \{\pm 1\}$
of $L$ with
\begin{equation}
  \epsilon(\alpha,\beta)\epsilon(\beta,\alpha)=(-1)^{(\alpha,\beta)}
\end{equation}
we can introduce the real twisted group ring $\R_\epsilon[L]$ of $L$ with
$e^\alpha e^\beta=\epsilon(\alpha,\beta)e^{\alpha+\beta}$ for $\alpha,\beta\in L$.
The Lie algebra $\hat\h$ acts on this space by $Ke^\alpha=0$ and
$h(n)e^\alpha=\delta_{n,0}(\alpha,h)e^\alpha$ for all $\alpha\in L$, $h\in\h$ and
$n\in \Z$. Then the underlying vector space of $V_L$ is given by
$S(\hat{\h}_-)\otimes \R_\epsilon[L]$ with its induced $\hat\h$-action. We denote
the operators $1\otimes e^\alpha$ by $e^\alpha$ and using that $K$ acts by $1$ on
$S(\hat{\h}_-)$, we obtain the relations 
\begin{align}
  [h(n),h'(m)]&=n\delta_{n+m,0}(h,h')\text{ and}\\
  [h(n),e^\alpha]&=\delta_{n,0}(\alpha,h)e^\alpha.
\end{align}
For $h\in\h$ we introduce the field
\begin{equation}
  h(z)=\sum_{n\in \Z}h(n)z^{-n-1}
\end{equation}
and for $\alpha\in L$ we define
\begin{equation}\label{Octopussy}
  \Gamma_\alpha(z)=e^{\alpha}z^{\alpha_0}\exp\left(-\sum_{n<0}\alpha(n)
  \frac{z^{-n}}{n}\right)\exp\left(-\sum_{n>0}\alpha(n)\frac{z^{-n}}{n}\right).
\end{equation}
By applying formal derivatives $\partial_z$ and normal ordering $:\cdot:$ we obtain
a vertex operator $Y$ on $V_L$ which satisfies
\begin{equation}
  Y(h(-1)\otimes e^0,z)=h(z)\quad \text{and}\quad
  Y(1\otimes e^\alpha,z)=\Gamma_\alpha(z)
\end{equation}
for all $h\in\h$ and $\alpha\in L$. Together with the vacuum $\vac=1\otimes e^0$,
the vertex operator $Y$ induces a vertex algebra structure on $V_L$. Notice that
different choices of the $2$-cocycle $\epsilon$ yield isomorphic vertex algebras.
We therefore fix a cocycle for the rest of this work.
  
Given $0\ne \alpha\in L$ and $k\in\Z$, $k\geq 0$, we introduce the 
\emph{Schur polynomials} $S_k(\alpha)$ by
\begin{equation*}
  \exp\left(-\sum_{n<0}\alpha(n)\frac{z^{-n}}{n}\right)=
  \sum_{k=0}^\infty S_k(\alpha)z^k.
\end{equation*}
We easily derive that for $k\in\Z$, $k\geq 0$, we have
\begin{equation}\label{Q}
  kS_k(\alpha)=\sum_{m=1}^k \alpha(-m)S_{k-m}(\alpha).
\end{equation}
All modes are of the form $\alpha(-m)$ for $m>0$, therefore all their commutators
vanish. We obtain
\begin{align}
  S_0(\alpha)=1,\,
  S_1(\alpha)=\alpha(-1)\text{ and }
  2S_2(\alpha)=\alpha(-1)^2+\alpha(-2).
\end{align} 
Of course the expression for the $0$-th Schur polynomial is a direct consequence
of the definition above.
  
Fix a basis $h_i$, $1\leq i\leq l$, of $\h$ and denote its dual basis by $h^i$.
The vector 
\begin{equation}
  \omega = \frac{1}{2}\sum_{i=1}^{l}h_i(-1)h^i(-1)\vac,
\end{equation}
which is independent of the choice of the basis $h_i$, is a conformal vector of
$V_L$ and the corresponding Virasoro field
\begin{equation}
  Y(\omega,z)=\sum_{n\in\Z}L_n z^{-n-2}=\frac{1}{2}\sum_{i=1}^l:h_i(z)h^i(z):
\end{equation}
turns $V_L$ into a conformal vertex algebra of central charge $l=\rk(L)$. Notice
that $V_L$ is in general not a vertex operator algebra. Yet this is the case if
the lattice $L$ is positive-definite. See \cite{Bo86} and \cite{KAC}.

For $0\ne \alpha\in L$ and $k,n\in\Z_{\geq 0}$ we have
\begin{equation}\label{TheSpangledMob}
  [L_{-n},S_k(\alpha)]= \sum_{i=1}^{k}\alpha(-n-i)S_{k-i}(\alpha).
\end{equation}	
This generalisation of (\ref{Q}) can be proved by induction over $k$.

We focus on the conformal lattice vertex algebra $V_{\hyp}$, associated with the
unique even, unimodular lattice $\hyp$ of signature $(1,1)$. This vertex algebra
can be equipped with a non-degenerate, invariant bilinear form $(\cdot,\cdot)$, 
normalised by $(\vac,\vac)=-1$. Furthermore $V_{\hyp}$ carries a natural vertex
algebra involution $\theta$, defined by $\theta(e^\alpha)=(-1)^{\alpha^2/2}(e^{\alpha})^{-1}$
and $\theta(h(n))=-h(n)$, for all $\alpha\in L$, $h\in\h$ and $n\in\Z$. Here the
inverse $(e^{\alpha})^{-1}$ has to be taken in the twisted group ring, i.e.
$(e^{\alpha})^{-1}=\epsilon(\alpha,-\alpha)e^{-\alpha}$. This involution preserves
the bilinear form. See \cite{Bo86}, \cite{Bo92} and \cite{Ju95}.
  
Let $V$ be a real vertex operator algebra of central charge $24$ of CFT-type with
a non-degenerate, invariant bilinear form $(\cdot,\cdot)$, normalised such that
$(\vac,\vac)=1$. Then $V\otimes V_{\hyp}$ is a conformal vertex algebra of central
charge $26$.  We denote its Virasoro modes by
\begin{equation}
  L(n)=L_n\otimes 1+1\otimes L_n
\end{equation}
in order to avoid misunderstandings. The bilinear forms on $V$ and $V_{\hyp}$
naturally extend to $V\otimes V_{\hyp}$ by
$(\cdot,\cdot)=(\cdot,\cdot)_V\otimes(\cdot,\cdot)_{V_{\hyp}}$.
(We indicate the respective vertex algebra by a subscript.)
This bilinear form is again invariant and  called the \emph{invariant bilinear form}.
Additionally we extend the involution $\theta$ to this space by $1\otimes\theta$
and we have a natural isometric $G_V$-action by letting $g\in G_V$ act as
$g\otimes 1$. Notice that the induced $\hyp$-grading, with weight spaces
\begin{equation*}
  \HS(\alpha)=V\otimes V_{\hyp,\alpha}=V\otimes S(\hat{\h}_-)\otimes e^\alpha
\end{equation*}
for $\alpha\in \hyp$, is compatible with the $L(0)$-grading. We introduce the
\emph{contravariant bilinear form}
\begin{equation*}
  (\cdot,\cdot)_0=-(\cdot,\theta(\cdot)),
\end{equation*}
on $V\otimes V_{\hyp}$, which is a symmetric, non-degenerate bilinear form. The
adjoint of $L(n)$ with respect to $(\cdot,\cdot)_0$ is given by $L(-n)$ and we have
\begin{equation}
  \HS(\alpha)\perp \HS(\beta)
\end{equation}
for $\alpha,\beta\in \hyp$ unless $\alpha=\beta$ with respect to this bilinear form. 
  
The space of \emph{primary states} (also \emph{physical states}) is given by
\begin{equation*}
  \PS=\{v\in V\otimes V_{\hyp}:L(n)v=0\text{ for all }n>0\}.
\end{equation*}
This space is homogeneous for the $\Z$-grading induced by $L(0)$ as well as the natural
$\hyp$-grading. Denote the corresponding weight spaces by $\PS^n$ and $\PS(\alpha)$,
for $n\in\Z$ and $\alpha\in\hyp$. Since those gradings are compatible we set
$\PS^n(\alpha)=\PS^n\cap \PS(\alpha)$. Standard properties of the vertex algebra
$V\otimes V_{\hyp}$ imply that the space $\PS^1/L(-1)\PS^0$ is a Lie algebra with the
Lie bracket $\big[[x],[y]\big]=[x_0y]$ for $x,y\in \PS^1$. Since $L(-1)\PS^0$ is
contained in the kernel of the bilinear form on $\PS^1$, this Lie algebra has a
natural invariant bilinear form $(\cdot,\cdot)$. 

The \emph{Lie algebra of physical states} is given by 
\begin{equation}
  \g(V)= \big(\PS^1/L(-1)\PS^0\big)/\ker(\cdot,\cdot)
\end{equation}
and equipped with a Lie algebra involution $\theta$, a non-degenerate, invariant
bilinear form $(\cdot,\cdot)$, a contravariant bilinear form $(\cdot,\cdot)_0$ and
a $\hyp$-grading. The respective weight spaces  $\g(V)_\alpha$ are given by the
decomposition
\begin{equation*}
  \g(V)=\PS^1/\ker (\cdot,\cdot)_0 = \bigoplus_{\alpha\in\hyp}\PS^1(\alpha)/
  \left(\ker (\cdot,\cdot)_0\cap\PS^1(\alpha)\right).
\end{equation*}
Since the action of the group $G_V$ preserves all the structures above, it acts naturally
on the Lie algebra of physical states and its weight spaces. This construction for the Lie
algebra of physical states is called the \emph{covariant quantisation}. Compare with the
similar discussion in \cite{Bo92} or \cite{Ju95}, where all details can be found.
  
Notice that an isomorphic Lie algebra can be constructed using BRST-cohomology.
For example this approach was taken in \cite{CKS}, \cite{M21} or \cite{DSSW}. Finally
we remark that such Lie algebras often are \emph{generalised Kac-Moody algebras} in
the sense of \cite{Bo88}.  Yet we do not need this fact for the following discussion.

\section{The no-ghost theorem}
In this section we review the no-ghost theorem, which is based on the work of physicists
Goddard and Thorn. It was first applied in representation theory by I. Frenkel.
Later it famously appeared in Borcherds' proof of the Monstrous moonshine conjecture.
Our exposition is based on \cite{Bo92}, \cite{Ju95} and \cite{Fr85}.
    
Let $V$ be a vertex operator algebra of central charge $24$ of CFT-type with a
non-degenerate, invariant bilinear form $(\cdot,\cdot)$. For the even, unimodular
lattice $\hyp$ of signature $(1,1)$ we set $\mathfrak{h}=\hyp\otimes\mathbb{R}$ and
consider the associated conformal lattice vertex algebra $V_{\hyp}$. In addition
we keep the notations from the previous section. 
  
Take $0\ne\alpha\in\hyp$. Fix an isotropic element $w_\alpha\in\h$
with $(\alpha,w_\alpha)=1$. For each  $n\in\mathbb{Z}$ we introduce the operator
$K_\alpha(n)=1\otimes w_\alpha(n)$, acting on $V\otimes V_{\hyp}$, which satisfies
\begin{equation}\label{LeChiffre}
  [K_\alpha(n),K_\alpha(m)]=0\text{ and }
  [L(n),K_\alpha(m)]=-mK_\alpha(n+m)
\end{equation}
for $n,m\in\mathbb{Z}$.  With respect to the contravariant bilinear form $(\cdot,\cdot)_0$,
the adjoint operator of $K_\alpha(n)$ is given by $K_\alpha(-n)$. See \cite{Bo92} and
\cite{Ju95}. On $\HS(\alpha)$ we have $K_{\alpha}(0)=1$ such that the corresponding
generating series is 
\begin{equation*}
  K_\alpha(z)=\sum_{n\in\Z} K_{\alpha}(n)z^{-n}= 1+K_{\alpha,0}(z).
\end{equation*}
This series has a formal inverse on $\HS(\alpha)$, given by
\begin{equation*}
  K_\alpha(z)^{-1}=(1+K_{\alpha,0}(z))^{-1}=1-K_{\alpha,0}(z)+K_{\alpha,0}(z)^2
  -K_{\alpha,0}(z)^3\pm\cdots.
\end{equation*}
Following \cite{Fr85} we denote its modes on $\HS(\alpha)$ by
\begin{equation*}
  K_\alpha(z)^{-1}=\sum_{n\in\Z} D_\alpha(n)z^{-n}.
\end{equation*}
For $n,m\in\Z$, $m>0$, we define operators
\begin{equation*}
  D_\alpha(n,m)=\sum_{\substack{n_1+\cdots+n_m=n \\ n_1,\cdots,n_m\ne 0}}
  K_\alpha(n_1)\cdots K_\alpha(n_m),
\end{equation*}
acting on $\HS(\alpha)$. In addition set $D_\alpha(n,0)=\delta_{n,0}$ for $n\in\Z$.
These sums yield well-defined operators, since on any element $v\in \HS(\alpha)$
just finitely many summands act non-trivially. Notice that for any $n\in\Z$ and
$v\in \HS(\alpha)$ we have $D_\alpha(n,m)v=0$ for sufficiently large $m$. We then obtain
\begin{equation}
  D_\alpha(n)=\sum_{m=0}^\infty (-1)^mD_\alpha(n,m)
\end{equation}
by direct computation. The same reasoning can be applied to check that the formal
inverse $K_\alpha(z)^{-1}$ yields well-defined operators on $\HS(\alpha)$.
For $k,n\in\Z$ we have
\begin{equation*}
  [L(k),D_\alpha(n)]=-(2k+n)D_\alpha(k+n).
\end{equation*}
Consider the operator
\begin{equation}
  E_\alpha=(D_\alpha(0)-1)(L(0)-1)+\sum_{n=1}^{\infty}
  (D_\alpha(-n)L(n)+L(-n)D_\alpha(n))
\end{equation}
on $\HS(\alpha)$, which plays a central role in the approach to the no-ghost theorem
presented in \cite{Fr85}. For $n\in\Z$ we have $[K_\alpha(n),E_\alpha]=-nK_\alpha(n)$
and furthermore
\begin{align}\label{JamesBond}
  \left[D_\alpha(n),E_\alpha\right]=-nD_\alpha(n)
  \text{ and } [L(n),E_\alpha]=-nL(n).
\end{align}
The proof of the third equation uses the fact that the Virasoro algebra acts
with central charge $26$ on $V\otimes V_{\hyp}$. 
  
For $v\in \HS(\alpha)$ and tuples $\lambda=(\lambda_1,\cdots,\lambda_n)$ and
$\mu=(\mu_1,\cdots,\mu_m)$ with $\lambda_i,\mu_j\in\Z_{\geq 0}$ we set
\begin{equation}
  v_{\lambda,\mu}=L(-1)^{\lambda_1}\cdots L(-n)^{\lambda_n}K_\alpha(-1)^{\mu_1}
  \cdots K_\alpha(-m)^{\mu_m}v.
\end{equation}
Except for the \emph{empty tuple} $\und{0}=()$, we will assume $\lambda_n\ne 0$
and $\mu_m\ne 0$. The space $\HS(\alpha)$ has some important subspaces:
\begin{enumerate}[(1)]
  \item $\KS(\alpha)$ is the subspace of $\HS(\alpha)$ annihilated by all
  $K_\alpha(n)$ for $n>0$.
  \item $\TS(\alpha)=\PS(\alpha) \cap \KS(\alpha)$ and
  $\TS^1(\alpha)=\PS^1(\alpha) \cap \KS(\alpha)$.
  \item $\GS(\alpha)$ is the span of all $t_{\lambda,\mu}$ for tuples $\lambda$
  and $\mu$ with $\sum_i \lambda_i+\sum_j \mu_j>0$ and $t\in \TS(\alpha)$.
  \item $\KS'(\alpha)$ is the subspace of $\GS(\alpha)$ generated by all
  $t_{\und{0},\mu}$ with $\mu\ne \und{0}$.
  \item $\SSp(\alpha)$ is the space of \emph{spurious} vectors in $\HS(\alpha)$,
  i.e. all vectors spanned by $t_{\lambda,\mu}$ with $\lambda\ne \und{0}$.
\end{enumerate}
For simplicity we write $Ve^{\alpha}=V\otimes \R e^{\alpha}$ and $N=\ker(\cdot,\cdot)_0$
for the kernel of the bilinear form of $\PS(\alpha)$. Spurious vectors are perpendicular
to $\PS(\alpha)$ and the operators $L(-m)$, $m>0$, preserve the space $\SSp(\alpha)$ of
such states.
\begin{lemma}[{Lemmas 5.1, 5.2 and 5.3 in \cite{Bo92}}]\label{VesperLynd}
  The restriction of the bilinear form $(\cdot,\cdot)_0$ to $\TS(\alpha)$ is
  non-degenerate and its kernel in $\KS(\alpha)$ is given by $\KS'(\alpha)$. 
  For any orthogonal basis $t_k$ of $\TS(\alpha)$ the vectors
  $(t_k)_{\lambda,\mu}$ form a basis of $\HS(\alpha)$. Furthermore:
  \begin{enumerate}[(1)]
    \item We have $\HS(\alpha)=\TS(\alpha)\oplus \GS(\alpha)$ and
	$\HS(\alpha)=\KS(\alpha)\oplus \SSp(\alpha)$.
	\item We also have $\KS(\alpha)=Ve^\alpha\oplus \KS'(\alpha)$, 
	$\KS(\alpha)=\TS(\alpha)\oplus \KS'(\alpha)$ and
	$\GS(\alpha)=\KS'(\alpha)\oplus \SSp(\alpha)$.
  \end{enumerate} 
\end{lemma}
\noindent
Proofs for these statements can be found in the literature cited above. See in
particular \cite{Ju95} for a detailed discussion.
  
From $D_\alpha(0)\vert_{\TS(\alpha)}=1$ we deduce $E_\alpha t=0$ for all
$t\in \TS(\alpha)$. The commutator relations in (\ref{JamesBond}) imply that
the vectors $t_{\lambda,\mu}$, with $t\in\TS(\alpha)$, are eigenvectors of the
operator $E_\alpha$, more precisely
\begin{equation}
  E_\alpha t_{\lambda,\mu}=-(\lambda_1+2\lambda_2+\cdots+n\lambda_n+
  \mu_1+2\mu_2+\cdots+m\mu_m)t_{\lambda,\mu}.
\end{equation}
We fix an orthogonal basis $t_k$ of $\TS(\alpha)$ and consider the corresponding
basis of eigenvectors $(t_k)_{\lambda,\mu}$ of $\HS(\alpha)$. Observe that most
of the subspaces of $\HS(\alpha)$ defined above have bases of eigenvectors of $E_\alpha$.
  
The direct decompositions in Lemma \ref{VesperLynd} induce projection maps
\begin{equation}\label{Goldfinger}
  p_{\TS}: \HS(\alpha)\to \TS(\alpha)\quad \text{and}\quad
  p_{V}: \HS(\alpha)\to Ve^\alpha,
\end{equation}
which both have kernel $\GS(\alpha)$. For simplicity we suppress the weight $\alpha$
for those maps. Clearly the restriction $p_\TS\vert_{Ve^\alpha}:Ve^\alpha\to\TS(\alpha)$
is a bijective linear map with inverse $p_V\vert_{\TS(\alpha)}$. Since $\TS(\alpha)$
is the $0$-eigenspace of $E_\alpha$, we can describe $p_\TS$ explicitly. In fact,
for any $v\in \HS(\alpha)$ we have
\begin{equation}\label{M}
  p_{\TS} v=\lim_{d\to\infty}\frac{1}{d!}\left(\prod_{i=1}^d(E_\alpha+i)\right)v.
\end{equation}
This limit makes sense because for any $v\in\HS(\alpha)$ the corresponding sequence
becomes stationary for large $d$. 
\begin{lemma}[{Lemma 5.5 in \cite{Bo92}}]\label{ReneMathis}
  The space $\PS^1(\alpha)$ is the direct sum of $\TS^1(\alpha)$ and $N^1$.
\end{lemma}
\begin{proof}
  We have to show that each $p\in \PS^1(\alpha)$ can be written as $t+n$ for unique
  elements $t\in \TS^1(\alpha)$ and $n\in N^1$. We have $p=k+s$ for $k\in \KS^1(\alpha)$
  and $s\in \SSp^1(\alpha)$. Since $\SSp(\alpha)$ is spanned by elements $t_{\lambda,\mu}$
  with $\lambda\ne 0$ and $t\in \TS(\alpha)$, it is clear that the operator $E_\alpha$
  preserves spurious states. This means $E_\alpha s\in \SSp(\alpha)$ for each
  $s\in \SSp(\alpha)$. For $p\in \PS^1(\alpha)$ we get
  \begin{equation*}
    E_\alpha p=\sum_{m=1}^\infty L(-m)D_\alpha(m)p\in \SSp(\alpha).
  \end{equation*}
  Together this implies $E_\alpha k=E_\alpha p-E_\alpha s\in \SSp(\alpha)$. Since
  $E_\alpha$ preserves $\KS(\alpha)$ this yields $E_\alpha k=0$ as we have
  $\SSp(\alpha)\cap \KS(\alpha)=\{0\}$. Hence $k\in \TS^1(\alpha)$ and we get
  $s=p-k\in \PS^1(\alpha)$ with $\TS^1(\alpha)\subset \PS^1(\alpha)$. Altogether
  $s\in \PS^1(\alpha)\cap \SSp^1(\alpha)$, which implies $s\in N^1$. Notice that the
  restriction of $(\cdot,\cdot)_0$ to $\TS^1(\alpha)$ is non-degenerate and
  so $\TS^1(\alpha)\cap N^1=\{0\}$.
\end{proof}
\noindent
This proof is similar to the one in \cite{Dr22} but differs from those given in
\cite{Bo92} or \cite{Ju95}. Yet for all these proofs it is crucial that the Virasoro
algebra acts with central charge $26$ on $V\otimes V_{\hyp}$. Here that fact is hidden
in the properties of the operator $E_\alpha$. See the remark after (\ref{JamesBond}). 
\begin{theorem}[no-ghost theorem]
  For every $0\ne\alpha\in\hyp$ the linear isomorphism
  \begin{equation*}
  	\eta_\alpha: V_{1-\alpha^2/2}\to \mathfrak{g}(V)_\alpha=\PS^1(\alpha)/N^1,\,
    v\mapsto [p_\TS(v\otimes e^\alpha)],
  \end{equation*}
  preserves the $G_V$-action and the bilinear forms. More precisely for all
  $v,w\in V_{1-\alpha^2/2}$ we have $(v,w)=(\eta_\alpha(v),\eta_\alpha(w))_0$.
  The linear isomorphism
  \begin{equation*}
    \eta_0: V_1\oplus\hyp\otimes\R \to \g(V)_{0},\, v+\alpha\mapsto v\otimes
    e^0+\vac\otimes \alpha(-1)e^0,
  \end{equation*}
  has the same properties, where $G_V$ acts trivially on $\hyp\otimes \R$.
\end{theorem}
\noindent
This version of the no-ghost theorem was formulated in \cite{Bo92}. The proof is a
combination of the lemmas above and an explicit calculation with the bilinear forms.
The maps $\eta_\alpha$ are called the \emph{no-ghost isomorphisms}.

\section{Lie brackets and vertex operators}
In this section we consider the Lie algebra of physical states $\mathfrak{g}(V)$
of a vertex operator algebra $V$ of central charge $24$ of CFT-type
with a non-degenerate, invariant bilinear form and apply the no-ghost theorem
to express its Lie bracket explicitly in terms of vertex algebra operations.
  
Let $V$ be a vertex operator algebra of central charge $24$ of CFT-type with a
non-degenerate, invariant bilinear form. We denote its Lie algebra of physical
states by $\mathfrak{g}(V)$. The no-ghost theorem allows us to identify weight spaces
of the vertex operator algebra $V$ with weight spaces of the $\hyp$-grading of $\g(V)$.  
This induces bilinear maps
\begin{align*}
  \{\cdot,\cdot\}_{\alpha,\beta}: V_{1-\alpha^2/2}\times V_{1-\beta^2/2}&\to 
  V_{1-(\alpha+\beta)^2/2}, \\(v,w)&\mapsto \epsilon(\alpha,\beta)\eta_{\alpha+\beta}^{-1}
  ([\eta_\alpha(v),\eta_\beta(w)]).
\end{align*}
We introduce the scaling $\epsilon(\alpha,\beta)$ to make the map
$\{\cdot,\cdot\}_{\alpha,\beta}$ independent of the choice of the cocycle $\epsilon$.

We fix a basis $e,f$ of the lattice $\hyp$ with $e^2=f^2=0$ and $(e,f)=1$. An
element $\alpha\in\hyp$ can be written as $\alpha=(\alpha,f)e+(\alpha,e)f$
and satisfies $\alpha^2/2=(\alpha,e)(\alpha,f)$. 
  
The discussion in the previous section showed that the no-ghost isomorphisms are not
entirely without ambiguity because they depend on a choice of isotropic element
$w_\alpha\in\h$ with $(\alpha,w_\alpha)=1$. We have to make these choices in a consistent way.
For any $\alpha\in\hyp$ with $(\alpha,f)\ne 0$ we set $w_\alpha=\frac{1}{(\alpha,f)}f$.
Therefore we restrict the evaluation of $\{\cdot,\cdot\}_{\alpha,\beta}$ to the case
$\alpha,\beta,\alpha+\beta\notin f^\perp$.
For all $n\in\Z$ the numbers  $x_{\alpha,\beta}=(w_\alpha,\beta)\in\Q$ satisfy
\begin{equation}\label{Moneypenny}
  K_\alpha(n)=x_{\alpha,\beta}K_\beta(n).
\end{equation}
In particular we have $x_{\alpha,\alpha}=1$ and $x_{\alpha,\beta}x_{\beta,\alpha}=1$.
 
For the Monster Lie algebra $\mathfrak{m}=\g(V^\natural)$, the above
assumptions impose no restrictions since there are no isotropic roots.
  
The first step in our examination will be to evaluate the expressions
\begin{equation}\label{Solitaire}
  p_{\TS}(v\otimes e^\alpha)_0p_{\TS}(w\otimes e^\beta)\mod \GS(\alpha+\beta)
\end{equation}
for all $v\in V_{1-\alpha^2/2}$ and $w\in V_{1-\beta^2/2}$  such that we can compute
their projections onto $Ve^{\alpha+\beta}$ explicitly.
  
For $k\in\Z_{\geq 0}$ and $j\in\{0,\cdots,k\}$ we define polynomials
$p^k_j\in\mathbb{Z}[T]$ by
\begin{enumerate}[(1)]
  \item $p^k_0=T^k$,
  \item $p^k_j=Tp^{k-1}_j+p^{k-1}_{j-1}$ for $j\in \{1,\cdots,k-2\}$,
  \item $p^k_{k-1}=kT$ and $p^k_k=1$.
\end{enumerate}
We give a generalisation of (\ref{JamesBond}) for the operators $D_\alpha(n)$.
\begin{lemma}\label{DrNo}
  Take $0\ne \alpha\in \hyp$. For  $k,n\in\Z$, $k\geq 0$, we have
  \begin{equation*}
    [D_\alpha(n),E_\alpha^k]=-\sum_{j=0}^{k-1}p^k_j(n)D_\alpha(n)E_\alpha^j.
  \end{equation*}
\end{lemma}
\begin{proof}
  Fix any $n\in\Z$. We give a proof by induction over $k$. For $k=0$ the statement is
  trivial and for $k=1$ the statement is given in (\ref{JamesBond}). By assumption
  the statement holds for $k$, hence   
  \begin{align*}
    [D_\alpha(n)&,E_\alpha^{k+1}]=
    [D_\alpha(n),E_\alpha]E_\alpha^k+E_\alpha[D_\alpha(n),E_\alpha^k]\\
    =&-p^1_0(n)D_\alpha(n)E_\alpha^k-\sum_{j=0}^{k-1}p^k_j(n)
    E_\alpha D_\alpha(n)E_\alpha^j\\
    =&-p^1_0(n)D_\alpha(n)E_\alpha^k-\sum_{j=0}^{k-1}p^k_j(n)(nD_\alpha(n)E_\alpha^j
    +D_\alpha(n)E_\alpha^{j+1})\\
    =&-p^1_0(n)D_\alpha(n)E^k_\alpha
    -\sum_{j=1}^{k-1}\left[p^k_j(n)nD_\alpha(n)E_\alpha^j+p^k_{j-1}(n)
    D_\alpha(n)E_\alpha^j\right]\\
    &-p^k_0(n)nD_\alpha(n)E_\alpha^0-p^k_{k-1}(n)D_\alpha(n)E_\alpha^k\\
    =&-p^k_0(n)nD_\alpha(n)E_\alpha^0
    -\sum_{j=1}^{k-1}\left(p^k_j(n)n+p^k_{j-1}(n)\right)D_\alpha(n)E_\alpha^j\\
    &-(p^k_{k-1}(n)+p^1_0(n))D_\alpha(n)E_\alpha^k\\
    =&-\sum_{j=0}^{k}p^{k+1}_j(n)D_\alpha(n)E_\alpha^j.
  \end{align*}
  This is the statement of the lemma for $k+1$.
\end{proof}
\noindent
We remark that this result holds for all isotropic vectors $w_\alpha\in\h$ which
satisfy $(w_\alpha,\alpha)=1$.
    
Some of the later results are obtained by taking certain sums over sets of tuples
of integers. We have already seen this in the definition of the operators $D_\alpha(n,m)$.
To simplify notations we make some further definitions. Take $m,n\in\Z$, $m>0$. Set
\begin{equation*}
  B^m(n)=\{\und{n}=(n_1,\cdots,n_m)\in\Z_{>0}^m:n_1+\cdots+n_m=n\}
\end{equation*}
and in addition set $B^0(n)=\emptyset$ for all $n\ne0$ and $B^0(0)=\{\und{0}\}$.
Recall that by $\und{0}$ we denote the empty tuple $()$. Clearly we have
$B^m(n)=\emptyset$, whenever $m>n$, so that
\begin{equation*}
  B(n)=\bigcup_{m=0}^\infty B^m(n)=\bigcup_{m=0}^{n}B^m(n).
\end{equation*}
The cardinalities of those sets will be relevant later. We introduce polynomials
\begin{equation*}
  b_n=\sum_{m=0}^n\vert B^m(n)\vert T^m,
\end{equation*}
which in particular satisfy $\vert B(n)\vert=b_n(1)$.
  
Since $V\otimes V_{\hyp}$ is a conformal vertex algebra we can consider the modes of
the element $\vac\otimes w_\alpha(-1)e^0$, i.e.
\begin{equation*}
  Y(\vac\otimes w_\alpha(-1)e^0,z)=\sum_{n\in\Z}(\vac\otimes w_\alpha(-1)e^0)(n)z^{-n-1}.
\end{equation*}
With the vacuum axiom $Y(\vac,z)=1$ we compute
\begin{equation}
  K_\alpha(n)=1\otimes w_\alpha(n)=(\vac\otimes w_\alpha(-1)e^0)(n).
\end{equation}
This simple (and well-known) observation allows us to apply the Borcherds identity
to those operators.
\begin{lemma}\label{Moonraker}
  Take $n,k\in\mathbb{Z}$, $n\geq 0$. For all $x\in \KS(\alpha)$ and
  $y\in \KS(\beta)$ we have
  \begin{align*}
    (D_\alpha(-n)x)_ky&=(-1)^n b_n(x_{\alpha,\beta})x_{k-n}y \mod \GS(\alpha+\beta)
    \text{ and}\\
	x_k(D_\beta(-n)y)&=b_n(x_{\beta,\alpha})x_{k-n}y \mod \GS(\alpha+\beta).
  \end{align*}
\end{lemma}
\begin{proof}
  In the case $n=0$ we have $D_{\alpha}(0)\vert_{\KS(\alpha)}=1$ and $b_0=1$ so that
  the statement is trivial. Consider $n>0$. By assumption we have
  $\alpha,\beta,\alpha+\beta\notin f^\perp$ so that $K_\alpha(j)=x_{\alpha,\beta}K_\beta(j)$
  and $K_\alpha(j)=x_{\alpha,\alpha+\beta}K_{\alpha+\beta}(j)$ for all $j\in\mathbb{Z}$. 
  Compare with (\ref{Moneypenny}). Using this and the Borcherds identity we get
  \begin{align*}
    (K_\alpha(&-n)x)_ky\\
	&=\sum_{j=0}^{\infty}(-1)^j\binom{-n}{j}
	(K_\alpha(-n-j)(x_{k+j}y)-(-1)^nx_{k-n-j}(K_\alpha(j)y))\\
	&=-(-1)^n x_{\alpha,\beta}x_{k-n}y+K_{\alpha+\beta}(-n)(\cdots)+
	K_{\alpha+\beta}(-n-1)(\cdots)+\cdots\\
	&= -(-1)^n x_{\alpha,\beta}x_{k-n}y\mod \GS(\alpha+\beta).
  \end{align*}
  The operators $K_\alpha(h)$ preserve the space $\KS(\alpha)$ for all $h\in\Z$.
  Therefore the commutator identity (\ref{LeChiffre}) implies
  \begin{align*}
    (D_\alpha(&-n,m)x)_ky\\
	&=\sum_{\substack{n_1+\cdots+n_m=n \\ n_1,\cdots,n_m> 0}}
	(K_\alpha(-n_1)\cdots K_\alpha(-n_m)x)_ky\\
	&=\sum_{\substack{n_1+\cdots+n_m=n \\ n_1,\cdots,n_m> 0}}(-1)^{m+n}x_{\alpha,\beta}^m
	x_{k-n}y\mod \GS(\alpha+\beta)\\
	&= (-1)^{m+n}x_{\alpha,\beta}^m \vert B^m(n)\vert x_{k-n}y\mod \GS(\alpha+\beta).
  \end{align*}
  Taking the sum over $m$ yields the stated result for $D_\alpha(n)$. The second equation
  can be proved analogously.                                                                                                                                                                                                                                                                                                                                                                                                                                                                                                                                                                                                                                                                                                                                                                                                                                                                                                                                                                                                                                                                                                                                              
\end{proof}

For all $k,n\in\Z_{\geq 0}$ we need certain maps $\imath_{k,n}$ and $\jmath_{k,n}$
acting on $V$. We define them recursively. We start with
\begin{align}\label{Hugo}
  \imath_{0,n}=\jmath_{0,n}=\delta_{0,n}
\end{align}
and for $k>0$ we set
\begin{align}
  \imath_{k,n}&=\sum_{j=1}^k\sum_{m=1}^n(-1)^mb_m(x_{\alpha,\beta})p_{j-1}^{k-1}(-m)
  \imath_{j-1,n-m}L_m\text{ and }\label{Drax}\\
  \jmath_{k,n}&=\sum_{j=1}^k\sum_{m=1}^nb_m(x_{\beta,\alpha})p_{j-1}^{k-1}(-m)\jmath_{j-1,n-m}L_m.
\end{align}
These maps are of degree $-n$, with respect to the $L_0$-grading of $V$. We find that
\begin{align*}
  \imath_{0,n}&=\delta_{0,n},\\
  \imath_{1,n}&=(-1)^n\delta_{n>0}b_n(x_{\alpha,\beta})L_n\text{ and}\\
  \imath_{2,n}&=(-1)^n\sum_{m=1}^{n-1}b_{n-m}(x_{\alpha,\beta})
  b_m(x_{\alpha,\beta})L_{n-m}L_m-(-1)^nnb_n(x_{\alpha,\beta})L_n,
\end{align*}
where
\begin{equation*}
\delta_{n>0}=
  \begin{dcases*}
    1 & if  $n>0$, \\[1ex]
    0 & if $n\leq 0$.
  \end{dcases*}
\end{equation*}
Notice that similar formulas can be derived for the maps $\jmath_{k,n}$ as well.
\begin{prop}
  For $k,h\in\Z$, $k\geq 0$, and all $v,w\in V$, $x\in \KS(\alpha)$ and $y\in \KS(\beta)$
  we have
  \begin{align*}
    \big(E_\alpha^k(v\otimes e^\alpha)\big)_h y&=\sum_{n=0}^\infty (\imath_{k,n}(v)\otimes
    e^{\alpha})_{h-n}y \mod \GS(\alpha+\beta)\text{ and}\\
	x_h\big(E_\beta^k(w\otimes e^\beta)\big)&=\sum_{n=0}^\infty
	x_{h-n}(\jmath_{k,n}(w)\otimes e^{\beta}) \mod \GS(\alpha+\beta).
  \end{align*}  
\end{prop}
\begin{proof}
  We give a proof by induction over $k$. The case $k=0$ is trivial. Using
  $v\otimes e^\alpha\in\KS(\alpha)$ and the Lemmas \ref{DrNo} and \ref{Moonraker} we find
  \begin{align*}
    \big(E&_\alpha^{k+1}v\otimes e^\alpha\big)_h y= \sum_{m=1}^\infty
	\big(E_\alpha^k D_\alpha(-m)(L_mv\otimes e^\alpha)\big)_hy\\
    =&\sum_{m=1}^\infty\big(\big([E_\alpha^k, D_\alpha(-m)]+D_\alpha(-m)E_\alpha^k\big)
    (L_mv\otimes e^\alpha)\big)_hy\\
    =&\sum_{m=1}^\infty\sum_{j=1}^{k}p^k_{j-1}(-m)\big(D_\alpha(-m)E_{\alpha}^{j-1}
    (L_m v\otimes e^\alpha)\big)_h y +\\
    &\sum_{m=1}^\infty \big(D_\alpha(-m)E_\alpha^k(L_mv\otimes e^\alpha)\big)_h y\\
    =&\sum_{m=1}^\infty\sum_{j=1}^{k+1}p^k_{j-1}(-m)\big(D_\alpha(-m)E_\alpha^{j-1}
    L_mv\otimes e^\alpha)\big)_{h}y\\
    =&\sum_{m=1}^\infty\sum_{j=1}^{k+1}(-1)^mb_m(x_{\alpha,\beta})p^k_{j-1}(-m)
    \big(E_\alpha^{j-1}(L_mv\otimes e^\alpha)\big)_{h-m}y\mod \GS(\alpha+\beta).
    \intertext{Assuming that the first statement of this proposition holds for all
    $0\leq j\leq k$, we can continue the computation with}
    =&\sum_{m=1}^\infty(-1)^mb_m(x_{\alpha,\beta})\sum_{j=1}^{k+1}p^k_{j-1}(-m)
    \sum_{l=0}^\infty\big((\imath_{j,l}L_mv)\otimes e^\alpha\big)_{h-m-l}y\\
    =&\sum_{n=0}^\infty\sum_{j=1}^{k+1}\sum_{m=1}^n(-1)^mb_m(x_{\alpha,\beta})
    p^k_{j-1}(-m)\big(\imath_{j,n-m}(L_mv)\otimes e^\alpha\big)_{h-n}y\\
    =&\sum_{n=0}^\infty \big(\imath_{k+1,n}(v)\otimes e^\alpha\big)_{h-n}y
    \mod \GS(\alpha+\beta),
  \end{align*}	
  where the last equation employed (\ref{Drax}). We proved the statement for $k+1$.
  The second identity follows analogously.
\end{proof}
  
So far our aim was to find an expression for (\ref{Solitaire}) such that we
can calculate its projection to $Ve^{\alpha+\beta}$ explicitly. As a
consequence of (\ref{M}), this product can be written as a linear combination of expressions
\begin{align*}
  \big(E_\alpha^{k_1}&(v\otimes e^\alpha)\big)_h
  \big(E_\beta^{k_2}(w\otimes e^\beta)\big)\\
  &=\sum_{n_1,n_2=0}^\infty(\imath_{k_1,n_1}(v)\otimes
  e^{\alpha})_{h-n_1-n_2} (\jmath_{k_2,n_2}(w)\otimes
  e^{\beta})\mod \GS(\alpha+\beta)
\end{align*}
for $k_1,k_2\in \Z_{\geq 0}$, $h\in \Z$ and $v,w\in V$. It remains to compute
the image of such linear combinations under $p_V$ (cf. (\ref{Goldfinger})).
    
For elements $\gamma,\delta\notin f^\perp$ we have $x_{\delta,\gamma}\ne 0$ and
introduce the numbers
\begin{equation}
  A(m)=-x_{\delta,\gamma}(1-x_{\delta,\gamma})^{m-1}
\end{equation}
for all $m\in\Z_{>0}$. Using $x_{\delta,\gamma}x_{\gamma,\delta}=1$ we compute
\begin{equation}\label{SMERSH}
  \sum_{i=1}^{m-1}A(i)=-\left(x_{\gamma,\delta}A(m)+1\right).
\end{equation}
Of course this equation already fully determines the numbers $A(m)$.
\begin{lemma}\label{Blades}
  We fix elements $\gamma,\delta\notin f^\perp$. For $k\in \Z$, $k\geq 0$,
  and $v\in V$ we have
  \begin{equation*}
    kv\otimes S_k(\gamma)e^{\delta} = -\sum_{m=1}^{k} A(m)v\otimes L_{-m}
    S_{k-m}(\gamma)e^{\delta}\mod \GS(\delta).
  \end{equation*}
\end{lemma}
\begin{proof}  
  The elements $\delta$ and $w_{\delta}$ clearly form a basis of $\hyp \otimes \R$
  and we have $\gamma=(\gamma,w_{\delta})\delta+cw_{\delta}$ for some constant
  $c\in \R$. Moreover the dual basis of $\delta,w_\delta$ is given by
  $w_\delta,\delta-\delta^2w_\delta$ and we  compute
  \begin{equation*}
    L_{-m}e^{\delta}=\delta(-m)e^{\delta}+w_{\delta}(-1)(\cdots)e^\delta+\cdots+
    w_{\delta}(-m)(\cdots)e^\delta.
  \end{equation*}
  This identity and $x_{\gamma,\delta}x_{\delta,\gamma}=1$ imply
  \begin{align*}
    v\otimes &S_{k-m}(\gamma)L_{-m}e^{\delta}\\
    &=v\otimes S_{k-m}(\gamma)(\delta(-m)e^{\delta}+w_{\delta}(-1)(\cdots)
    e^\delta\cdots+w_{\delta}(-m)(\cdots)e^\delta)\\
    &=v\otimes S_{k-m}(\gamma)\delta(-m)e^{\delta}+K_{\delta}(-1)(\cdots)+
    \cdots+K_{\delta}(-m)(\cdots)\\
    &=v\otimes S_{k-m}(\gamma)\delta(-m)e^{\delta} \mod \GS(\delta)\\
    &=x_{\gamma,\delta}v\otimes S_{k-m}(\gamma)\gamma(-m)e^{\delta}
    \mod \GS(\delta).
  \end{align*}
  Applying this identity and (\ref{TheSpangledMob}) we find
  \begin{align*}
    \sum_{m=1}^{k}& A(m)v\otimes L_{-m}S_{k-m}(\gamma)e^{\delta}\\
	=&\sum_{m=1}^{k} A(m)v\otimes \left(S_{k-m}(\gamma)L_{-m}e^{\delta}+
	\sum_{h=1}^{k-m}\gamma(-m-h)S_{k-m-h}(\gamma)e^\delta\right)\\
	=&-\sum_{m=1}^{k}v\otimes S_{k-m}(\gamma)\gamma(-m)e^{\delta}\\
	&+\sum_{m=1}^{k} \left(x_{\gamma,\delta}A(m)+1\right)v\otimes
	S_{k-m}(\gamma)\gamma(-m)e^{\delta}\\
	&+\sum_{m=1}^{k}\sum_{h=1}^{k-m}A(m)v\otimes\gamma(-m-h)S_{k-m-h}
	(\gamma)e^\delta\mod \GS(\delta).\\
  \end{align*}
  By (\ref{Q}), the first sum is the expression we seek. It is enough
  to show that the remaining sums add up to $0\mod \GS(\delta)$. A rearrangement
  of the indices and (\ref{SMERSH}) yield
  \begin{align*}
    &\sum_{m=1}^{k} \left(x_{\gamma,\delta}A(m)+1\right)v\otimes
    S_{k-m}(\gamma)\gamma(-m)e^{\delta}\\
    &+\sum_{m=1}^{k}\sum_{h=1}^{k-m}A(m)v\otimes\gamma(-m-h)S_{k-m-h}
    (\gamma)e^\delta \\
    =&\sum_{m=1}^{k} \left(x_{\gamma,\delta}A(m)+1\right)v\otimes
    S_{k-m}(\gamma)\gamma(-m)e^{\delta}\\
    &+\sum_{i=2}^{k}\left(\sum_{j=1}^{i-1}A(j)\right)v\otimes\gamma(-i)
    S_{k-i}(\gamma)e^\delta=0\mod \GS(\delta).
  \end{align*}
  The last equality uses $x_{\gamma,\delta}A(1)+1=0$.
\end{proof}
\noindent
We remark that this lemma also holds if $(\gamma,w_{\delta})=0$, since both
sides of the identity vanish. Yet we are mostly interested in the case
$\gamma=\alpha$ and $\delta=\alpha+\beta$, for which we always have
$x_{\alpha+\beta,\alpha}\ne 0$.
\begin{lemma}\label{SaratogaSprings}
  We fix elements $\gamma,\delta\notin f^\perp$. Then for $k\in \Z_{\geq 0}$
  and $v\in V$ we have
  \begin{equation*}
    kp_V(v\otimes S_k(\gamma)e^{\delta})=\sum_{m=1}^{k} A(m)p_V(L_{-m}v\otimes
    S_{k-m}(\gamma)e^{\delta}).
  \end{equation*}
\end{lemma}
\begin{proof}
  We have $\sum_{m=1}^kA(m)L(-m)(v\otimes S_{k-m}(\gamma)e^\delta)=0\mod \GS(\delta)$.
  With Lemma \ref{Blades} we obtain
  \begin{align*}
    \sum_{m=1}^{k}& A(m)L_{-m}v\otimes S_{k-m}(\gamma)e^{\delta}\\
    &=-\sum_{m=1}^{k} A(m)v\otimes L_{-m}S_{k-m}(\gamma)e^{\delta}
    =kv\otimes S_k(\gamma)e^{\delta}  \mod \GS(\delta).
  \end{align*}
  Since the kernel of $p_V$ is precisely $\GS(\delta)$, the lemma is proved.
\end{proof}
\noindent
This lemma allows us to compute the projection $p_V(v\otimes S_k(\gamma)e^{\delta})$
by a recursive argument. 
  
Given a tuple $\und{n}=(n_1,\cdots,n_m)\in B^m(n)$ we put $-\und{n}=(-n_1,\cdots,-n_m)$
and introduce the numbers
\begin{equation*}
  c(\und{n})=\prod_{i=1}^m\frac{A(n_i)}{n_i+\cdots+n_m}.
\end{equation*}
Furthermore set $L(-\und{n})=L_{-n_m}\cdots L_{-n_1}$. In addition we set $L(-\und{0})=1$
and $c(\und{0})=1$ for the empty tuple $\und{0}=()$. It will sometimes be convenient to
shift indices in the set $B(n)$ such that
\begin{equation*}
  B^{m+1}(n)=\{(n_0,\und{n}'):n_0\in\{1,\cdots,n-m\}\text{, }\und{n}'\in B^{m}(n-n_0)\}.
\end{equation*}
For $\und{n}=(n_0,\und{n}')\in B^{m+1}(n)$ we then have $L(-\und{n})=L(-\und{n}')L_{-n_0}$
and $c(\und{n})=\frac{A(n_0)}{n}c(\und{n}')$. We introduce the operators
\begin{equation*}
  p^{\gamma,\delta}_n=\sum_{\und{n}\in B(n)}c(\und{n})L(-\und{n}).
\end{equation*}
The numbers $c(\und{n})$ depend on $\gamma$ and $\delta$ because the numbers $A(m)$ do.
Using the recursive description of the numbers $c(\und{n})$ and the sets $B(n)$ we find a
recursive identity for the operators $p^{\gamma,\delta}_n$. This is
\begin{equation}\label{GoldenEye}
  np^{\gamma,\delta}_n=\sum_{j=1}^nA(j)p^{\gamma,\delta}_{n-j}L_{-j}.
\end{equation}
We calculate $p^{\gamma,\delta}_0=1$, $p^{\gamma,\delta}_1=A(1)L_{-1}$ and
$2p^{\gamma,\delta}_2=A(1)^2L_{-1}^2+A(2)L_{-2}$.
\begin{prop}\label{MI6}
  For $k\in\Z$, $k\geq 0$, and $v\in V$ we have
  $p_V(v\otimes S_k(\gamma)e^\delta)=p^{\gamma,\delta}_kv$.
\end{prop}
\begin{proof}
  We give a proof by induction. The case $k=0$ is trivial. Assume the statement
  for all $j\leq k$. Then Lemma \ref{SaratogaSprings} and the identity (\ref{GoldenEye})
  imply
  \begin{align*}
    (k+1)p_V(v\otimes &S_{k+1}(\gamma)e^\delta)\\
    &=\sum_{m=1}^{k+1}A(m)p_V(L_{-m}v\otimes S_{k+1-m}(\gamma)e^{\delta})\\
	&=\sum_{m=1}^{k+1}A(m)p^{\gamma,\delta}_{k+1-m}(L_{-m}v)
	=(k+1)p^{\gamma,\delta}_{k+1}v.
  \end{align*}	
  This is the statement for the case $k+1$.
\end{proof}

We set $n=1-\alpha^2/2$ and $m=1-\beta^2/2$. Consider $v\in V_{n}$ and
$w\in V_{m}$. Since the degree of elements in $\HS(\alpha)$ is bounded from below
by $\alpha^2/2$, we find
\begin{equation}
  p_{\TS}(v\otimes e^\alpha)=\sum_{i=1}^n\frac{1}{n!}
  S_i(1,\cdots,n)E_\alpha^i(v\otimes e^\alpha).
\end{equation}
Denoting the coefficients in those sums by $S_{i,n}$, we may introduce a
weighted sum of the operators $\imath_{k,n}$ and $\jmath_{k,n}$ as follows
\begin{align}\label{GoldenGun}
  \imath^{(n)}_{n_1}=\sum_{k_1=1}^nS_{k_1,n}\imath_{k_1,n_1}
  \text{ and }
  \jmath^{(m)}_{n_2}=\sum_{k_2=1}^mS_{k_2,m}\jmath_{k_2,n_2}.
\end{align}
These operators act on $V_n$ or $V_m$, respectively. In the following we will drop
the upper indices to simplify notations.
\begin{theorem}\label{Blofeld}
  Let $V$ be a vertex operator algebra of central charge $24$ of CFT-type with a
  non-degenerate, invariant bilinear form. Choose a primitive isotropic element
  $f\in\hyp$ and assume that $\alpha,\beta,\alpha+\beta\notin f^\perp$. Then for
  $v\in V_{1-\alpha^2/2}$ and $w\in V_{1-\beta^2/2}$ we have
  \begin{align*}
    \{v,w\}_{\alpha,\beta}=
    \sum_{n_1,n_2=0}^\infty\sum_{k=0}^\infty p^{\alpha,\alpha+\beta}_{n_1+n_2+k-(\alpha,\beta)}
    (\imath_{n_1}(v)_k\jmath_{n_2}(w)).
   \end{align*}
\end{theorem}
\begin{proof}
  With the operators $\imath_{n_1}$ and $\jmath_{n_2}$, defined in (\ref{GoldenGun}),
  we find
  \begin{align*}
    p_{\TS}(v\otimes &e^\alpha)_0p_{\TS}(w\otimes e^\beta)\\
  	=&\sum_{n_1,n_2=0}^\infty\big(\imath_{n_1}(v)\otimes e^\alpha\big)_{-n_1-n_2}
  	\big(\jmath_{n_2}(w)\otimes e^\beta\big)\mod \GS(\alpha+\beta).
  \end{align*}
  Using (\ref{Octopussy}) we obtain for any $m\in\Z$ and $x,y\in V$ that
  \begin{align*}
    \big(x\otimes e^\alpha\big)_{-m}\big(y\otimes e^\beta\big)
  	=\epsilon(\alpha,\beta) \sum_{n}x_n y\otimes S_{n+m-(\alpha,\beta)}(\alpha)
  	e^{\alpha+\beta}.
  \end{align*}
  Hence by Proposition \ref{MI6} and the fact that we rescaled the map
  $\{\cdot,\cdot\}_{\alpha,\beta}$ by $\epsilon(\alpha,\beta)$, the theorem follows.
\end{proof}
  
We consider a simple example of this result. Let $V$ be a vertex operator algebra as
in Theorem \ref{Blofeld} and take $\alpha\in\hyp$ with $1-\alpha^2/2=2$. Then
$\mathfrak{g}(V)_\alpha$ is isomorphic to $V_2$ under a no-ghost isomorphism.
For primary states $v,w\in V_2$, we deduce
\begin{align*}
  \{v,w\}_{\alpha,\alpha}
  =&\frac{1}{2}(v_{-2}w-w_{-2}v)-\frac{1}{8}(L_{-2}+L_{-1}^2)(v_0w-
  \frac{1}{2}L_{-1}v_1w)\\
  &+\frac{1}{128}(L_{-2}^2-2L_{-2}L_{-1}^2-\frac{7}{3}L_{-1}^4)v_{2}w.
\end{align*}
  
The case where $V$ is a real form of a unitary vertex operator algebra $V_\C$
of central charge $24$ with character
\begin{equation*}
  \text{ch}_{V_\C}(\tau)=j(\tau)-744 =q^{-1}+0+196884q+\cdots
\end{equation*}
is of particular interest. Here $j$ is Klein's $j$-invariant and $q=e^{2\pi i\tau}$.
In this case Frenkel, Lepowsky and Meurman \cite{FLM} conjectured that $V_\C$
is isomorphic to the (complex) Moonshine module, which is often called the
\emph{FLM conjecture}. Following the arguments in \cite{Bo92}, we can show that the
associated Lie algebra of physical states $\g(V)$ is a generalised Kac-Moody
algebra and isomorphic to the Monster Lie algebra. This observation has been applied by
Carnahan \cite{Car17} and Miyamoto \cite{Mi23} to shed new light on the FLM conjecture.
More precisely Miyamoto computed the $\Z$-graded vector space structure of Zhu's
Poisson algebra $V_\C/C_2(V_\C)$ of the vertex operator algebra $V_\C$. As a consequence the
$C_2$-cofiniteness of such vertex operator algebras follows.
Furthermore the restriction of Theorem \ref{Blofeld} to this special case answers
Question $4$, posed by Borcherds in Section $15$ of \cite{Bo92}.

\end{document}